\documentclass[a4paper,10pt]{amsart} 

\usepackage[ansinew]{inputenc} 
\usepackage[T1]{fontenc}
\usepackage[english]{babel} 

\usepackage{enumerate} 
\usepackage{booktabs}

\usepackage{amsmath} 
\usepackage{amsfonts} 
\usepackage{amssymb}
\usepackage{amsthm}
\usepackage{latexsym}
\usepackage{bbm}
\usepackage{a4wide}
\usepackage[pdftex]{graphicx}
  \DeclareGraphicsExtensions{.pdf,.png,.jpg,.jpeg,.mps}
  \usepackage{pgf}
  \usepackage{tikz}

\makeatletter

\@namedef{subjclassname@2010}{%

  \textup{2010} Mathematics Subject Classification}

\makeatother

\theoremstyle{plain} 
\newtheorem{theorem}{Theorem} 
\newtheorem{lemma}{Lemma} 
\newtheorem{prop}{Proposition} 
\newtheorem*{theorem*}{Theorem}
\newtheorem*{kk}{Khintchine-Kahane Inequality}
\newtheorem*{contraction}{Kahane's Contraction Principle}

\theoremstyle{definition} 

\newtheorem*{definition*}{Definition}

\theoremstyle{remark}
 
\newtheorem*{remark}{Remark} 

\newcommand{\R}{\mathbbm{R}}

\newcommand{\Z}{\mathbbm{Z}}
\newcommand{\C}{\mathbbm{C}}
\newcommand{\E}{\mathbbm{E}}
\newcommand{\K}{\mathbbm{K}}
\newcommand{\prob}{\mathbbm{P}}
\newcommand{\D}{\,\text{d}}

\title[Carleson's embedding and the maximal theorem]{On the relation of Carleson's embedding and the maximal theorem in the context of Banach space geometry}
\author[T. Hytönen and M. Kemppainen]{Tuomas Hytönen and Mikko Kemppainen}
\address{Department of Mathematics and Statistics, University of Helsinki,
Gustaf Hällströmin katu 2b, FI-00014 Helsinki, Finland}

\email{tuomas.hytonen@helsinki.fi, mikko.k.kemppainen@helsinki.fi}

\begin{document}




\maketitle

\begin{abstract}
Hyt\"onen, McIntosh and Portal (J. Funct. Anal., 2008) proved two vector-valued generalizations of the classical Carleson embedding theorem, both of them requiring the boundedness of a new vector-valued maximal operator, and the other one also the type $p$ property of the underlying Banach space as an assumption. We show that these conditions are also necessary for the respective embedding theorems, thereby obtaining new equivalences between analytic and geometric properties of Banach spaces.
\end{abstract}


\section{Introduction}

Let $E_j$ denote the averaging operator with respect to the dyadic cubes of sidelength $2^{-j}$ in~$\R^n$. The classical Carleson embedding theorem, in its dyadic version, characterizes the sequences $(\theta_j)_{j\in\Z}$ of functions $\theta_j\in L^2_{\textup{loc}}(\R^n)$ for which the map $f\mapsto(E_j f\cdot\theta_j)_{j\in\Z}$ embeds $L^2(\R^n)$ boundedly into $L^2(\Z\times\R^n)=L^2(\R^n;\ell^2)$. The usual proofs show that this embedding theorem is a corollary of the (dyadic) maximal inequality in $L^2(\R^n)$. The present article shows, in a more general context and among other things, that the two theorems are actually equivalent in a precise sense to be described. But first we give some background to motivate our considerations.

In the treatment of an infinite-dimensional version of the famous Kato square root problem (related to the functional calculus of elliptic divergence form operators), Hyt\"onen, McIntosh and Portal~\cite{HMP} encountered the need of a Carleson embedding for functions  $f\in L^p(\R^n;E)$  (the Bochner $L^p$ space with values in the Banach space~$E$). The relevant variant for the mentioned application involved replacing the classical sequence space $\ell^2$ appearing in the scalar version by the space $\textup{Rad}(E)$ of almost unconditionally summable sequences in $E$, which, of course, is no surprise to experts in vector-valued Harmonic Analysis. Thus there was a need to obtain reasonable conditions for the boundedness of the embedding $f\mapsto(E_j f\cdot\theta_j)_{j\in\Z}$ from $L^p(\R^n;E)$ to $L^p(\R^n;\textup{Rad}(E))$. Note that, already for $E=\C$, this led to apparently new considerations involving embeddings of the type $L^p(\R^n)\to L^p(\R^n;\ell^2)$, which are different from (and more difficult than) the straightforward $L^p$ generalizations of the classical embedding to $L^p(\R^n)\to L^p(\R^n;\ell^p)$.

In order to carry out an argument somewhat reminiscent of the classical proof, the authors of~\cite{HMP} introduced a new maximal operator $M_R$ for vector-valued functions. They then deduced two versions of the vector-valued Carleson embedding theorem under the condition that the maximal inequality $\|M_R f\|_{L^p(\R^n)}\leq C\|f\|_{L^p(\R^n;E)}$ holds for $f\in L^p(\R^n;E)$. This condition is satisfied by many classical Banach spaces~$E$ such as all reflexive $L^q$ spaces (and even their noncommutative counterparts), but not for instance by $E=\ell^1$.  Thus this maximal inequality defines a nontrivial Banach space property, which was termed RMF (for Rademacher maximal function) in \cite{HMP} and further studied by Kemppainen~\cite{RMF}.

Concerning the two versions of the vector-valued embedding, recall that Carleson's classical theorem gives an exact characterization of the admissible sequences $(\theta_j)_{j\in\Z}$ in terms of the so-called Carleson condition. There is an analogous condition $\text{Car}^p$ for every $p\in(1,\infty)$, which is easily seen to be necessary for the embedding of $L^p(\R^n;E)$. Its sufficiency was established in \cite{HMP} under the assumption that $E$ has the RMF property and so-called \emph{type} $p$, a well-established notion from the Geometry of Banach Spaces. Without the type $p$ assumption, the embedding was only obtained under a stronger Carleson condition $\text{Car}^{p+\epsilon}$ with $\epsilon>0$.

While both the RMF and the type $p$ assumptions where somewhat \emph{ad hoc} at the time of writing~\cite{HMP}, being basically the assumptions needed to make the particular method of proof work, it is the purpose of this paper to show that both these conditions are actually necessary for the respective embedding theorems. On the one hand, this gives further justification for the relevance of RMF as a new class of Banach spaces. On the other hand, the necessity of type is already interesting for the scalar-valued case $E=\C$, as no Banach space can have type $p>2$. This limits the optimal embedding theorem with $\text{Car}^p$ (rather than $\text{Car}^{p+\epsilon}$) to the spaces $L^p(\R^n)$ with $p\in(1,2]$. For quite a while, the first-named author believed that one should be able to take $\epsilon=0$ for all $p\in(1,\infty)$, until Michael Lacey provided him with a counterexample when $p=4$ (personal communication, September 2009). It was soon clear that this could be extended to all $p>2$, and this eventually led to the abstract result in the context of Banach spaces as formulated in this paper.

The RMF property also played a role in an earlier version of the characterization of the boundedness of vector-valued singular integral operators with respect to nonhomogeneous measures by Hyt\"onen~\cite{HYTONENNONHOMTB}, although this assumption was eventually eliminated from the final version of that paper. A variant of the vector-valued Carleson embedding theorem is still used there, but the point is that  the RMF assumption can be dispensed with provided that the functions $\theta_j$ satisfy the additional condition that $\theta_j=E_j\theta_j$. Without such additional structure, however, the RMF is equivalent to the Carleson embedding, as we show here.

In companion with the results of Kemppainen~\cite{RMF}, we now know various analytic conditions equivalent to the vector-valued maximal inequality. It is still an open question, however, to describe it in terms some established notions from the Geometry of Banach Spaces. In particular, it would be interesting to know if the important UMD property is sufficient for RMF.

\section{Preliminaries}

All Banach spaces can be either real or complex unless otherwise stated and so we speak of scalars without 
specifying whether they are real or complex. The scalar field, either $\R$ or $\C$, is generically denoted by $\K$.

We write $a \lesssim b$ when there exists a constant $C$ such that $a \leq Cb$, with $C$ 
independent of the indicated variables in expressions $a$ and $b$. By $a\eqsim b$ we mean $b\lesssim a \lesssim b$. 
Isomorphism of Banach spaces is denoted by $\simeq$. Sets of vectors indexed by a subset of a larger index set are
always thought to have zero extension to the whole index set.

Let $(\varepsilon_j)_{j=1}^{\infty}$ be a sequence of \emph{Rademacher variables}, more precisely, a
sequence of independent random variables attaining values $-1$ and $1$ with an equal probability
$\prob (\varepsilon_j = -1) = \prob (\varepsilon_j = 1) = 1/2$. We write $\E$ for the corresponding expectation.

The following technique of \emph{randomization} will be used at times in order to handle randomized norms.
If $(\varepsilon_j)_{j=1}^N$ and $(\varepsilon'_j)_{j=1}^N$ are independent sequences of Rademacher variables,
then for any vectors $x_1, \ldots , x_N$ in a Banach space, the sequences $(\varepsilon_j x_j)_{j=1}^N$ and
$(\varepsilon'_j \varepsilon_j x_j)_{j=1}^N$ are identically distributed. In practise this is often applied in the
following way: if $\{ 1, \ldots , N \}$ is decomposed into disjoint sets $J_1, \ldots , J_M$, then
\begin{equation*}
  \E \Big\| \sum_{j=1}^N \varepsilon_j x_j \Big\| ^p 
  = \E \E ' \Big\| \sum_{k=1}^M \varepsilon'_k \sum_{j\in J_k} \varepsilon_j x_j \Big\| ^p ,
\end{equation*}
where $\E '$ denotes the expectation for $\varepsilon'_j$'s and $1\leq p < \infty$.

The following two standard results will be used frequently (see Kahane \cite{KAHANE} for proofs):

\begin{contraction}
  \label{contraction}
  Let $1\leq p < \infty$ and suppose that $x_1, \ldots , x_N$ are vectors in a Banach space. Then
  \begin{equation*}
    \E \Big\| \sum_{j=1}^N \varepsilon_j \lambda_j x_j \Big\| ^p \leq \Big( 2\max_{1\leq j \leq N} | \lambda_j | \Big) ^p
    \E \Big\| \sum_{j=1}^N \varepsilon_j x_j \Big\| ^p 
  \end{equation*}
  for any scalars $\lambda_1,\ldots , \lambda_N$.
  If the scalars $\lambda_j$ are real, the constant $2$ may be omitted.
\end{contraction}


\begin{kk}
\label{kkineq}
  For any $1 \leq p,q < \infty$, there exists a constant $K_{p,q}$ such that
  \begin{equation*}
    \Big( \E \Big\| \sum_{j=1}^N \varepsilon_j x_j \Big\|^p \Big)^{1/p}
    \leq K_{p,q} \Big( \E \Big\| \sum_{j=1}^N \varepsilon_j x_j \Big\|^q \Big)^{1/q},
  \end{equation*}
  whenever $x_1, \ldots , x_N$ are vectors in a Banach space.
\end{kk}

We recall the following fact concerning randomized series (see e.g. Diestel, Jarchow and Tonge \cite{DJT}, Theorem 12.3):
for a sequence $(x_j)_{j=1}^{\infty}$ of vectors in a Banach space $E$, the series
$\sum_{j=1}^{\infty} \varepsilon_j x_j$
converges almost surely if and only if it converges in $L^p$ for one (or equivalently, for each) $p\in [1,\infty )$.
Such sequences are called \emph{almost unconditionally summable}. The space of all these sequences in $E$
is denoted by $\text{Rad}(E)$ and when equipped with any of the equivalent norms
\begin{equation*}
  \Big\| (x_j)_{j=1}^{\infty} \Big\|_{\text{Rad}_p(E)} 
  = \Big( \E \Big\| \sum_{j=1}^{\infty} \varepsilon_j x_j \Big\| ^p \Big) ^{1/p}, \quad 1 \leq p < \infty ,
\end{equation*}
it becomes a Banach space.

\begin{remark}\
\begin{enumerate}
\item Although the sequences $(x_j)_{j=1}^{\infty}$ in $\text{Rad}(E)$ are not in general 
unconditionally summable, the sequences $(\varepsilon_j x_j)_{j=1}^{\infty}$ of random variables are unconditionally summable in 
the $L^p$-norm for any  $p\in [1,\infty )$. Thus the space $\text{Rad}(E)$ remains the same for different orderings of the index set.
\item For any Hilbert space $H$, there holds $\text{Rad} (H) = \ell^2(H)$, which is easy to check using the $\textup{Rad}_2$ norm and the orthogonality of the signs $\varepsilon_j$.
\item Kahane's Contraction Principle will often be applied in order to bound a finite sum by an infinite sum in the randomized norms.
\end{enumerate}
\end{remark}

In order to deduce the membership in $\textup{Rad}(E)$ of an infinite sequence from uniform estimates on its subsequences, we will need the following classical result of Kwapie\'n~\cite{KWAPIENC0}: 

\begin{prop}\label{prop:Kwapien}
If a Banach space $E$ does not contain an isomorphic copy of $c_0$ as a subspace, then for all sequences $(x_j)_{j=1}^{\infty}$ in $E$ there holds
\begin{equation*}
  \sup_{N\in\Z_+}\big\|(x_j)_{j=1}^N\|_{\textup{Rad}(E)}<\infty\qquad\Rightarrow\qquad
  (x_j)_{j=1}^{\infty}\in\textup{Rad}(E).
\end{equation*}
\end{prop}

The concept of type of a Banach space is intended to measure how far the randomized norms are from square sums of norms.
As we will prove, it also governs the form of Carleson's embedding theorem which one can obtain in a given Banach space.

\begin{definition*}
  A Banach space is said to have 
  \emph{type} $p \in [1,2]$ if there exists a constant $C$ such that
  \begin{equation*}
    \Big( \E \Big\| \sum_{j=1}^N \varepsilon_j x_j \Big\|^2 \Big)^{1/2}
    \leq C \Big( \sum_{j=1}^N \| x_j \|^p \Big)^{1/p}
  \end{equation*}
  for any vectors $x_1, \ldots , x_N$, regardless of $N$. 
%
\end{definition*}

\begin{remark}\ 
\begin{enumerate}
\item
Every Banach space has type $1$; hence we say that
a Banach space has \emph{nontrivial type}
if it has type $p$ for some $p > 1$.

\item
It follows from standard inequalities of $\ell^p$-norms that
if a space has type $p$ then it also has type $\tilde{p}$ when
$1\leq \tilde{p} \leq p$.

\item
One can show that
$L^p$-spaces have type $\min \{ p,2 \}$ when $1\leq p < \infty$.
Sequence spaces $\ell^1$ and $c_0$, on the other hand, are typical examples of spaces with only trivial type.

\item
Hilbert spaces have type $2$ with constant $C=1$ and equality of the randomized and quadratic norms.
\end{enumerate}
\end{remark}





In many questions of vector-valued Harmonic Analysis 
the uniform bound of a family of operators has to be replaced by its R-bound, first formally defined
by Berkson and Gillespie \cite{BERKSONGILLESPIE}.
The usefulness of this notion became widely recognized after its role in the seminal work of Weis~\cite{WEIS},
and it also lies behind the definition of the Rademacher maximal function, which we discuss in the following section.

\begin{definition*}
  A family $\mathcal{T} \subset \mathcal{L}(F,E)$ of linear operators from a Banach
  space $F$ to a Banach space $E$ is said to be \emph{R-bounded} if there
  exists a constant $C$ such that for any $T_1,\ldots , T_N \in \mathcal{T}$ and any 
  $x_1,\ldots , x_N \in F$, regardless of $N$, we have
  \begin{equation*}
    \E  \Big\| \sum_{j=1}^N \varepsilon_j T_j x_j \Big\|^p 
    \leq C^p \E \Big\| \sum_{j=1}^N \varepsilon_j x_j \Big\|^p ,
  \end{equation*}
  for some $p\in [1, \infty )$. 
  The smallest such constant is denoted by $\mathcal{R}_p (\mathcal{T})$. We denote $\mathcal{R}_2$ by $\mathcal{R}$
  for short later on.
\end{definition*}

Basic properties of R-bounds can be found for instance in Clément et al. \cite{CLEMENT}. We wish only to remark that
by the Khintchine-Kahane inequality, the R-boundedness of a family does not depend on $p$, and the constants
$\mathcal{R}_p (\mathcal{T})$ are comparable. R-bounds are (usually strictly) stronger than uniform norm bounds. They coincide with uniform bounds if $E$ and $F$ are Hilbert spaces.





\section{The Rademacher maximal function}

Suppose from now on that $F$ and $E$ are Banach spaces and that $\mathcal{X}\subset \mathcal{L}(F,E)$ is a Banach space
whose norm dominates the operator norm. Moreover, we require that $\mathcal{X}$ contains all the elementary tensors
\begin{equation*}
   f^*\otimes e:y\in F\mapsto f^*(y)e\in E,\qquad e\in E,\quad f^*\in F
\end{equation*}
and that $\|f^*\otimes e\|_{\mathcal{X}}=\|f^*\|_{F^*}\|e\|_E$. Fixing $f^*\in F^*$ or $e\in E$ of unit norm, this implies in particular that $\mathcal{X}$ contains an isometric copy of both $E$ and $F^*$.

Let $(\Omega , \mathcal{F} , \mu )$ be a $\sigma$-finite measure space and
denote the corresponding Lebesgue-Bochner
space of $\mathcal{F}$-measurable $\mathcal{X}$-valued functions by 
$L^p(\mathcal{F} ; \mathcal{X})$ (or $L^p(\mathcal{X})$), $1\leq p \leq \infty$.
The space of 
strongly measurable functions $f$ for which
$1_A f$ is integrable for every set $A\in\mathcal{F}$ with finite measure, is denoted by 
$L_{\textup{fin}}^1(\mathcal{F} ; \mathcal{X})$.

If $\mathcal{G}$ is a sub-$\sigma$-algebra of $\mathcal{F}$ such that $(\Omega, \mathcal{G}, \mu )$ is
$\sigma$-finite, there exists for every function $f\in L_{\textup{fin}}^1(\mathcal{F} ; \mathcal{X})$
a \emph{conditional expectation} 
$\E (f | \mathcal{G}) \in L_{\textup{fin}}^1 (\mathcal{G} ; \mathcal{X})$ with respect to $\mathcal{G}$
which is the (almost everywhere) unique strongly $\mathcal{G}$-measurable function satisfying
\begin{equation*}
  \int_A \E (f | \mathcal{G}) \D\mu = \int_A f \D\mu
\end{equation*}
for every $A\in\mathcal{G}$ with finite measure. The operator $\E ( \cdot | \mathcal{G} )$ is a contractive 
projection from $L^p(\mathcal{F} ; \mathcal{X})$ onto $L^p(\mathcal{G} ; \mathcal{X})$ for any $p\in [1,\infty ]$.
This follows immediately, if the vector-valued conditional expectation is constructed as the tensor extension of the
scalar-valued conditional expectation, which is a positive operator (see Stein
\cite{STEIN} for the scalar-valued case).

Suppose then that $(\mathcal{F}_j)_{j\in\Z}$ is a \emph{filtration}, that is, 
an increasing sequence of sub-$\sigma$-algebras of $\mathcal{F}$ such that each $(\Omega, \mathcal{F}_j, \mu )$ is
$\sigma$-finite. For a function $f\in L_{\textup{fin}}^1(\mathcal{F} ; \mathcal{X})$, 
we denote the conditional expectations with respect to this filtration by
\begin{equation*}
  E_j f := \E (f | \mathcal{F}_j) , \quad j\in\Z .
\end{equation*}
The standard maximal function (with respect to $(\mathcal{F}_j)_{j\in\Z}$) is given by
\begin{equation*}
  Mf(\xi ) = \sup_{j\in\Z} \| E_jf(\xi ) \| , \quad \xi\in\Omega .
\end{equation*}
The operator $f \mapsto Mf$ is known to be bounded from 
$L^p(\mathcal{X})$ to $L^p$ whenever $1<p\leq\infty$, regardless of $\mathcal{X}$.
The following variant was originally defined by Hytönen, McIntosh and Portal \cite{HMP} and later studied
in more detail by Kemppainen \cite{RMF}.

\begin{definition*}
The \emph{Rademacher maximal function} of a function $f\in L_{\textup{fin}}^1(\mathcal{F} ; \mathcal{X})$ is defined by
\begin{equation*}
  M_Rf(\xi ) = \mathcal{R} \Big( E_jf(\xi ) : j\in\Z \Big) , \quad \xi\in\Omega .
\end{equation*}
\end{definition*}

\begin{remark}

By the properties of R-bounds we obtain the pointwise relation
$Mf \leq M_Rf$. If $F$ and $E$ are Hilbert spaces, then $M_Rf=Mf$.
\end{remark}

In \cite{HMP}, Hytönen, McIntosh and Portal used the identification
$\mathcal{L}(\K , E) \simeq E$ and studied the Rademacher maximal function in the Euclidean case, where $\Omega = \R^n$ 
is equipped with Lebesgue measure and the filtration that is generated by dyadic cubes
$\mathcal{D}_j = \{ 2^{-j} ([0,1)^n + m) : m\in\Z^n \}$, $j\in\Z$. 
They showed that the $L^p$-boundedness of $f\mapsto M_Rf$ for one $p\in (1,\infty )$ implies boundedness
of a linearized version of $M_R$ both from $H^1$ to $L^1$ and from $L^{\infty}$ to $\text{BMO}$
and hence allows to interpolate in order to acquire boundedness between
Lorentz spaces $L^{p,s}$ for all $1 < p < \infty$, $1\leq s \leq \infty$ (see Hunt \cite{HUNTLORENTZ}
for details on interpolation between Lorentz spaces). 
They also provided an example of a space, namely $\ell^1$, for which
the Rademacher maximal operator is not bounded.

Kemppainen \cite{RMF} gave the definition in the above generality and showed that the boundedness of $M_R$ is independent of the filtration and the
underlying measure space in the following sense: the boundedness with respect to the filtration of dyadic intervals
on $[0,1)$ guarantees boundedness with respect to any filtration on any $\sigma$-finite measure space. This motivates the definition:

\begin{definition*}
A Banach space $\mathcal{X}\subset \mathcal{L}(F,E)$ is said to have RMF
if the Rademacher maximal operator with respect to the filtration of dyadic intervals
on $[0,1)$ is bounded from $L^p(\mathcal{X})$ to $L^p$ for one (or equivalently, for each) $p\in (1,\infty )$.
\end{definition*}

\begin{remark}\
\begin{enumerate}
\item The RMF-property is inherited by closed subspaces. In particular, from the assumption that $\mathcal{X}\subset\mathcal{L}(F,E)$ contain the elementary tensors $e\otimes f^*$, it follows that if $\mathcal{X}$ has RMF, then so do $E\simeq\mathcal{L}(\K,E)$ and $F^*=\mathcal{L}(F,\K)$.

\item Based on the fact that $\ell^1$ does not have RMF, it was shown by Kemppainen \cite{RMF} that if $\mathcal{X}$, and hence $E$, has RMF, then $E$  has some nontrivial type $p>1$. (The result was formulated for $\mathcal{X}=\mathcal{L}(F,E)$, but only used the fact that $E$ is isomorphic to a subspace of $\mathcal{X}$.) Such a space cannot contain an isomorphic copy of $c_0$, and hence Proposition~\ref{prop:Kwapien} is applicable in this situation.

\item When speaking of the RMF-property of $\mathcal{X}$, we always understand that the indentification of $\mathcal{X}$ as a subspace of an operator space $\mathcal{L}(F,E)$ has been fixed. If a space $E$ has no obvious operator structure, we always understand that the identification $E\simeq\mathcal{L}(\K,E)$ is used. The RMF-property does depend on the chosen identification! In particular, if $H$ and $K$ are infinite-dimensional Hilbert spaces, then $\mathcal{X}=\mathcal{L}(H,K)$ has RMF when viewed as $\mathcal{L}(H,K)$ (trivially, since then $M_Rf\leq Mf$), but it does not have RMF when viewed as $\mathcal{L}(\K,\mathcal{X})$ (since this would require that $\mathcal{X}$ have nontrivial type, and it does not). 
\end{enumerate}
\end{remark}


\section{Carleson's embedding theorem}



Recall that $F$ and $E$ are Banach spaces and $\mathcal{X}\subset\mathcal{L}(F,E)$ has the properties assumed in the beginning of the previous section.
Let $(\mathcal{F}_j)_{j\in\Z}$ be a filtration on a $\sigma$-finite measure space $(\Omega, \mathcal{F}, \mu )$.

\begin{definition*}
A family $\theta = (\theta_j )_{j\in\Z}$ of strongly $\mu$-measurable 
$F$-valued functions is called a \emph{$p$-Carleson family} for $p\in [1,\infty )$ if 
\begin{enumerate}
\item $(\theta_j(\xi))_{j\geq m}$ is in $\text{Rad}(F)$ for all $m\in\Z$ and $\mu$-almost every $\xi\in\Omega$,
\item there exists a constant $C$ such that for any integer $m$ and all sets $A\in\mathcal{F}_m$ we have
\begin{equation*}
  \int_A \E \Big\| \sum_{j\geq m} \varepsilon_j \theta_j (\xi ) \Big\|^p \D\mu (\xi ) \leq C^p \mu (A) .
\end{equation*}
\end{enumerate}
The smallest such constant is called the \emph{$p$-Carleson constant} $\| \theta \|_{\text{Car}^p}$ of $\theta$.
\end{definition*}

Observe that $\| \theta \|_{\text{Car}^p} \leq \| \theta \|_{\text{Car}^q}$ whenever $p\leq q$. This definition was introduced by Hyt\"onen, McIntosh and Portal~\cite{HMP} in the case of scalar-valued functions and the dyadic filtration of $\R^n$. The above generalization appears in~\cite{HYTONENNONHOMTB}, where it was used in the context of singular integrals with respect to nonhomogeneous measures.

Let $1 < p < \infty$ and $1 \leq s < \infty$.
The norm of the Lorentz space $L^{p,s}(\mathcal{X})$ on $(\Omega , \mu )$ is given by
\begin{equation*}
  \| f \|_{L^{p,s}(\mathcal{X})} = 
  \Big( \int_0^{\infty} \Big( \lambda 
  \mu ( \{ \xi\in\Omega : \| f(\xi ) \| > \lambda \} )^{1/p} ) \Big)^s \frac{\D\lambda}{\lambda} \Big)^{1/s} .
\end{equation*}
Recall that $\| f \|_{L^{p,s_2}} \leq \| f \|_{L^{p,s_1}}$ for $s_1\leq s_2$ and
$\| f \|_{L^{p,p}} \eqsim \| f \|_{L^p}$.

The following main lemma contains the heart of the deduction of Carleson's embedding theorem from the maximal inequality. It is based on the same stopping time technique as its original special case in Hyt\"onen, McIntosh and Portal~\cite{HMP}.

\begin{lemma}
\label{mainlemma}
  Let $1 < p < \infty$, suppose that $E$ has type $r\in [1,2]$ and write $s = \min \{ p,r \}$.
  For any $p$-Carleson family $\theta = (\theta_j )_{j\in\Z}$ we have
  \begin{equation*}
    \Big( \int_{\Omega} \sup_{N\in\Z}\E \Big\| \sum_{j\geq N} \varepsilon_j E_jf(\xi ) \theta_j(\xi ) \Big\|^p \D\mu (\xi ) \Big)^{1/p}
    \lesssim \| \theta \|_{\textup{Car}^p} \| M_R f \|_{L^{p,s}} ,
  \end{equation*}
whenever $f\in L^{p,s}(\mathcal{X})$ is such that the right side is finite. If $E$ does not contain an isomorphic copy of $c_0$, then the series on the left converges as $N\to-\infty$ for a.e.~$\xi$, and we also have
\begin{equation*}
    \Big( \int_{\Omega} \E \Big\| \sum_{j\in\Z} \varepsilon_j E_jf(\xi ) \theta_j(\xi ) \Big\|^p \D\mu (\xi ) \Big)^{1/p}
    \lesssim \| \theta \|_{\textup{Car}^p} \| M_R f \|_{L^{p,s}}.
    \end{equation*}  
\end{lemma}
  
\begin{proof}
By the contraction principle, the supremum $\sup_{N\in\Z}$ may be replaced by $\lim_{N\to-\infty}$, and then by Fatou's lemma it suffices to prove a similar statement with the supremum outside the integral. So we may consider $N\in\Z$ fixed, and then prove the assertion with a bound independent of $N$.
  
Let $f\in L^{p,s}(\mathcal{X})$.
    We may assume with no loss of generality that $M_Rf < \infty$ $\mu$-almost everywhere
    so that $(E_jf(\xi)\theta_j(\xi))_{j\geq N}$ is in $\text{Rad}(E)$ for $\mu$-almost every $\xi\in\Omega$.
    In order to break the sum into suitable pieces we define the stopping times
    \begin{equation*}
      \tau_k(\xi ) = \min \Big\{ j\geq N : \mathcal{R} \Big(  E_if(\xi ) : i\leq j \Big) > 2^k \Big\} , \quad
      k\in\Z , \quad \xi\in\Omega .
    \end{equation*}
    Since $M_Rf$ is finite $\mu$-almost everywhere, we have for $\mu$-almost every $\xi\in\Omega$ that
    $\tau_k(\xi ) = \infty$ when $k$ is big enough. On the other hand, 
    $\tau_k(\xi )$ may for some $\xi\in\Omega$ tend to some
    $\tau_{-\infty} (\xi ) > N$ as $k\to -\infty$, but then
    \begin{equation*}
      \sup_{j < \tau_{-\infty}(\xi )} \mathcal{R} \Big( E_if(\xi ) : i\leq j \Big) \leq 2^k
    \end{equation*}
    for all $k\in\Z$, which is possible only if $E_jf(\xi ) = 0$ whenever $j < \tau_{-\infty}(\xi )$.
  
    The set of indices $j\geq\tau_{-\infty}(\xi)$ can now be written as a union of finitely many disjoint sets
    \begin{equation*}
      J_k(\xi) = \{ j\in\Z : \tau_k (\xi) \leq j < \tau_{k+1} (\xi) \}
    \end{equation*} 
    at each point $\xi\in\Omega$.
    Using randomization and type $s = \min \{ p,r \}$ of $E$ we get
    \begin{align*}
      \E \Big\| \sum_{j\geq N} \varepsilon_j E_jf(\xi ) \theta_j(\xi ) \Big\|^p
      &= \E \E' \Big\| \sum_{k\in\Z} \varepsilon_k' \sum_{j\in J_k(\xi)}
        \varepsilon_j E_jf(\xi ) \theta_j(\xi ) \Big\|^p \\
      &\lesssim \Big( \sum_{k\in\Z} \E \Big\| \sum_{j\in J_k(\xi)}
        \varepsilon_j E_jf(\xi ) \theta_j(\xi ) \Big\|^s \Big)^{p/s} ,
    \end{align*}
    where
    \begin{equation*}
      \E \Big\| \sum_{j\in J_k(\xi)}
        \varepsilon_j E_jf(\xi ) \theta_j(\xi ) \Big\|^s
      \leq 2^{(k+1)s} \E \Big\| \sum_{j\in J_k(\xi)}
        \varepsilon_j \theta_j(\xi ) \Big\|^s 
    \end{equation*}
    by the definition of the stopping times $\tau_k$. Since $s\leq p$, we may use the triangle inequality in $L^{p/s}$ to get
    \begin{align*}
      \Big( \int_{\Omega} \E \Big\| \sum_{j\geq N} \varepsilon_j E_jf(\xi) \theta_j (\xi) \Big\|^p \D\mu (\xi) \Big)^{s/p}
      &\lesssim \Big( \int_{\Omega} \Big( \sum_{k\in\Z} 2^{(k+1)s} 
         \E \Big\| \sum_{j\in J_k(\xi)}
        \varepsilon_j \theta_j(\xi ) \Big\|^s  \Big)^{p/s} \D\mu (\xi) \Big)^{s/p}\\
      &\leq \sum_{k\in\Z} 2^{(k+1)s} \Big( \int_{\Omega} 
      \Big( \E \Big\| \sum_{j\in J_k(\xi)}
        \varepsilon_j \theta_j(\xi ) \Big\|^s  \Big)^{p/s} \D\mu (\xi) \Big)^{s/p},
    \end{align*}
    where
    \begin{equation*}
      \Big( \E \Big\| \sum_{j\in J_k(\xi)}
        \varepsilon_j \theta_j(\xi ) \Big\|^s  \Big)^{p/s} \eqsim
      \E \Big\| \sum_{j\in J_k(\xi)}
        \varepsilon_j \theta_j(\xi ) \Big\|^p
    \end{equation*}
    by the Khintchine-Kahane inequality.
    
    We write $A_m = \{ \xi\in\Omega : \tau_k(\xi ) = m \}$ for a fixed $k$ to split the space as
    $\Omega = \bigcup_{m\geq N} A_m$, where the value $m=\infty$ is a priori included in the union. Note that $A_m\in\mathcal{F}_m$
    for each integer $m\geq N$, and for $m=\infty$ the sum over $j\geq m$ is empty. Hence
    \begin{equation*}
      \int_{\Omega} \E \Big\| \sum_{j\in J_k(\xi)}
        \varepsilon_j \theta_j(\xi ) \Big\|^p \D\mu (\xi )
      \leq \sum_{m\geq N} 
      \int_{A_m} \E \Big\| \sum_{j\geq m} \varepsilon_j \theta_j (\xi ) \Big\|^p \D\mu (\xi ),
    \end{equation*}
    where the summation can be restricted to finite values of $m$, as usual.
    Using the $p$-Carleson condition for sets $A_m$ 
    we obtain
    \begin{equation*}
      \sum_{m\geq N} \int_{A_m} \E \Big\| \sum_{j\geq m} \varepsilon_j \theta_j (\xi ) \Big\|^p \D\mu (\xi )
      \leq \| \theta \|_{\text{Car}^p}^p \sum_{m\geq N} \mu (A_m)
      = \| \theta \|_{\text{Car}^p}^p \mu ( \{ \xi\in\Omega : \tau_k(\xi ) < \infty \} ) .
    \end{equation*}
    Observe that $\tau_k(\xi ) < \infty$ exactly when 
    $\mathcal{R} \Big( E_if(\xi ) : i\leq j \Big) > 2^k$ for some integer $j$, i.e. when
    $M_Rf(\xi ) > 2^k$. In conclusion,   
    \begin{align*}
      \Big( \int_{\Omega} \E \Big\| 
        \sum_{j\geq N} \varepsilon_j E_jf (\xi ) \theta_j (\xi ) \Big\|^p \D\mu (\xi ) \Big)^{s/p}
      &\leq \sum_{k\in\Z} 2^{(k+1)s} 
            \Big( \int_{\Omega} \E \Big\| \sum_{j\in J_k(\xi)}
            \varepsilon_j \theta_j (\xi ) \Big\|^p \D\mu (\xi ) \Big)^{s/p} \\
      &\leq \| \theta \|_{\text{Car}^p}^s 
        \sum_{k\in\Z} 2^{(k+1)s} \mu ( \{ \xi\in\Omega : M_Rf(\xi ) > 2^k \} )^{s/p} \\
      &\eqsim \| \theta \|_{\text{Car}^p}^s \| M_Rf \|_{L^{p,s}}^s .
    \end{align*}
This completes the proof of the case involving the truncated sums over $j\geq N$.

Let us then assume, in addition, that $E$ does not contain a copy of $c_0$. By the first part of the proof, we already know that
\begin{equation*}
  \sup_{N\in\Z}\E \Big\| 
        \sum_{j\geq N} \varepsilon_j E_jf (\xi ) \theta_j (\xi ) \Big\|^p<\infty\qquad\text{for a.e. }\xi\in\Omega
\end{equation*}
By Proposition~\ref{prop:Kwapien}, this implies that $(E_jf(\xi)\theta_j(\xi))_{j\in\Z}$ belongs to $\textup{Rad}(E)$ for all the $\xi$, and we may hence pass to the limit $N\to-\infty$ to obtain the second assertion.
 \end{proof}
 

Let $1 \leq p,q < \infty$ and suppose that $\theta = (\theta_j)_{j\in\Z}$ is a $q$-Carleson family of $F$-valued functions.
For every $f\in L_{\textup{fin}}^1(\mathcal{X})$ we define
\begin{equation*}
  \Theta f(\xi) = \Big( E_jf(\xi)\theta_j(\xi) \Big)_{j\in\Z} , \quad \xi\in\Omega .
\end{equation*}
We ask if the linear operator $\Theta$ is bounded from $L^p(\mathcal{X})$ to $L^p(\text{Rad}(E))$ and further if
the \emph{$(q,p)$-Carleson map} (with respect to the given filtration on the given $\sigma$-finite measure space)
\begin{equation*}
  \text{Car}^q \to \mathcal{L}\Big( L^p(\mathcal{X}) , L^p(\text{Rad}(E)) \Big) : \quad \theta \mapsto \Theta
\end{equation*}
is well-defined and bounded.


Now we come to the first main theorem.

\begin{theorem}
\label{mainthm1}
  Suppose that $\mathcal{X}\subset\mathcal{L}(F,E)$ is a Banach space 
  and let $1 < p < q < \infty$. The following conditions are equivalent:
  \begin{enumerate}
    \item\label{qpCarGeneral} The $(q,p)$-Carleson map 
          with respect to any filtration on any $\sigma$-finite measure space is well-defined and bounded.
    \item\label{qpCarDyadic} The $(q,p)$-Carleson map 
          with respect to the filtration of dyadic intervals on $[0,1)$ is well-defined and bounded.
    \item\label{XhasRMF} $\mathcal{X}$ has RMF.
  \end{enumerate}
\end{theorem}  

Note that $\eqref{XhasRMF}\Rightarrow\eqref{qpCarDyadic}$ was shown in \cite{HMP}. We will extend this argument to show the implication $\eqref{XhasRMF}\Rightarrow\eqref{qpCarGeneral}$; a similar claim was formulated without proof in \cite{HYTONENNONHOMTB}. This extension relies implicitly on the result of Kemppainen~\cite{RMF} that the RMF property with respect to the dyadic filtration already implies the corresponding property for arbitrary filtrations and measure spaces. The implication $\eqref{qpCarDyadic}\Rightarrow\eqref{XhasRMF}$ is completely new.
  
\begin{proof}
     $\eqref{qpCarGeneral}\Rightarrow\eqref{qpCarDyadic}$: Clear

    $\eqref{qpCarDyadic}\Rightarrow\eqref{XhasRMF}$: Take any positive integer $N$ and let
    $f\in L^p([0,1) ; \mathcal{X})$. There exists for every $\xi\in [0,1)$ elements
    $x^{(k)} = (x_j^{(k)})_{j=0}^N$, $k\in\Z_+$, of $\text{Rad}(F)$ such that
    \begin{equation*}
      \E \Big\| \sum_{j=0}^N \varepsilon_j x_j^{(k)} \Big\|^p \leq 1
    \end{equation*}
    and
    \begin{equation*}
      \E \Big\| \sum_{j=0}^N \varepsilon_j E_j f(\xi ) x_j^{(k)} \Big\|^p \to
      \mathcal{R}_p \Big( E_j f(\xi ) : 0\leq j \leq N \Big) ^p
    \end{equation*}
    as $k$ tends to infinity. Since each $E_jf$ is constant on intervals of $\mathcal{D}_N$, we only need to choose
    $2^N$ different $(x^{(k)})_{k=1}^{\infty}$'s, one for each interval. Thus we may define
    $\theta_j^{(k)}(\xi ) = x_j^{(k)}$, where $x^{(k)}$ corresponds to the interval containing $\xi$.
    It is immediate that each $\theta_j^{(k)}(\xi )$ is strongly measurable and that each
    $\theta^{(k)} = (\theta_j^{(k)})_{j=0}^N$ is a $q$-Carleson family with $\| \theta^{(k)} \|_{\text{Car}^q} \leq 1$.
    Thus
    \begin{equation*}
      \int_0^1 \mathcal{R}_p \Big( E_j f(\xi) : 0\leq j \leq N \Big)^p \D\xi
      \leq \liminf_{k\to\infty} \int_0^1 \E \Big\| \sum_{j=0}^N
      \varepsilon_j E_j f (\xi ) \theta_j^{(k)}(\xi ) \Big\|^p \D\xi
      \lesssim \| f \|_{L^p([0,1) ; \mathcal{X})}^p,
    \end{equation*}
    and consequently $M_R$ is bounded (remember that $\mathcal{R}_p$ and $\mathcal{R}_2$ -bounds are comparable).

    $\eqref{XhasRMF}\Rightarrow\eqref{qpCarGeneral}$: Suppose that $\mathcal{X}$ has RMF and let
    $\theta$ be a $q$-Carleson family with respect to any filtration on any $\sigma$-finite measure space.
    The RMF property implies in particular that $E$ cannot contain a copy of $c_0$, and hence Lemma~\ref{mainlemma} is applicable in its stronger form.
    Combining it with the fact that the Rademacher maximal operator maps $L^{q,1}(\mathcal{X})$
    boundedly to $L^{q,1}$ and $E$ has type 1 (trivially), Lemma \ref{mainlemma} shows that $\Theta f$ is well-defined for $f\in L^{q,1}(\mathcal{X})$ and
    \begin{equation*}
      \| \Theta f \|_{L^q(\text{Rad}(E))} \lesssim \| \theta \|_{\text{Car}^q} \| M_Rf \|_{L^{q,1}}
      \lesssim \| \theta \|_{\text{Car}^q} \| f \|_{L^{q,1}(\mathcal{X})} .
    \end{equation*}
    On the other hand, if $0 < \varepsilon < p - 1$, then $\theta$ is also a $(p - \varepsilon)$-Carleson family
    and a similar application of Lemma \ref{mainlemma} gives
    \begin{equation*}
      \| \Theta f \|_{L^{p-\varepsilon}(\text{Rad}(E))} 
      \lesssim \| \theta \|_{\text{Car}^{p-\varepsilon}} \| M_Rf \|_{L^{p-\varepsilon , 1}}
      \lesssim \| \theta \|_{\text{Car}^q} \| f \|_{L^{p-\varepsilon , 1}(\mathcal{X})} .
    \end{equation*}
    Hence $\Theta$ is bounded both from $L^{q,1}(\mathcal{X})$ to $L^q(\text{Rad}(E))$ and
    from $L^{p-\varepsilon , 1}(\mathcal{X})$ to $L^{p-\varepsilon}(\text{Rad}(E))$, which means that we may interpolate
    to get boundedness from $L^p(\mathcal{X})$ to $L^p(\text{Rad}(E))$. The $(q,p)$-Carleson map is thus
    well-defined and bounded.
  \end{proof}

\begin{theorem}
\label{mainthm2}
  Suppose that $\mathcal{X}\subset\mathcal{L}(F,E)$ is a Banach space 
  and let $1 < p < \infty$. The following conditions are equivalent:
  \begin{enumerate}
    \item\label{ppCarGeneral} The $(p,p)$-Carleson map 
          with respect to any filtration on any $\sigma$-finite measure space is well-defined and bounded.
    \item\label{ppCarDyadic} The $(p,p)$-Carleson map 
          with respect to the filtration of dyadic intervals on $[0,1)$ is well-defined and bounded.
    \item\label{XhasRMFandType} $\mathcal{X}$ has RMF and $E$ has type $p$.
  \end{enumerate}
\end{theorem}

Again, $\eqref{XhasRMFandType}\Rightarrow\eqref{ppCarDyadic}$ was shown in \cite{HMP}, and we extend this to $\eqref{XhasRMFandType}\Rightarrow\eqref{ppCarGeneral}$. The proof of $\eqref{ppCarDyadic}\Rightarrow\eqref{XhasRMFandType}$ was inspired by the counterexample by Michael Lacey, which demonstrated that the $(4,4)$-Carleson map with respect to the dyadic intervals is unbounded even for $\mathcal{X}=E=F=\R$.

  \begin{proof}
    $\eqref{ppCarGeneral}\Rightarrow\eqref{ppCarDyadic}$: Clear.

    $\eqref{ppCarDyadic}\Rightarrow\eqref{XhasRMFandType}$:
    The proof that $\mathcal{X}$ has RMF is identical to the corresponding
    argument in the proof of Theorem \ref{mainthm1}. For the claim on type $p$ of $E$, recall that the space $E\simeq\mathcal{L}(\K,E)$ itself has RMF under the assumptions, so we may take $F=\K$.
    Thus we assume that for any (scalar) $p$-Carleson family $\theta = (\theta_j)_{j=0}^{\infty}$ we have
    \begin{equation*}
      \Big( \int_0^1 \E \Big\| \sum_{j=0}^{\infty} \varepsilon_j \theta_j(\xi ) E_jf(\xi) \Big\| ^p \D\xi \Big)^{1/p}
      \lesssim \| \theta \|_{\text{Car}^p} \| f \|_{L^p(E)}
    \end{equation*}
    whenever $f$ is in $L^p(E)$.
    
    Suppose we are given $x_1,\ldots , x_N$ in $E$. We aim to construct a function $f\in L^p(E)$
    and a $p$-Carleson family $\theta$ for which
    \begin{equation*}
      \int_0^1 \E \Big\| \sum_{j=0}^{\infty} \varepsilon_j \theta_j(\xi) E_jf(\xi) \Big\|^p \D\xi
      = \E \Big\| \sum_{j=1}^N \varepsilon_j x_j \Big\|^p
      \quad \text{and} \quad \| f \|_{L^p(E)}^p \lesssim \sum_{j=1}^N \| x_j \|^p .
    \end{equation*}
    Further, an upper bound for $\| \theta \|_{\text{Car}^p}$ must not depend on 
    $N$ (nor on the vectors $x_1, \ldots , x_N$), from which it will follow by our
    assumption on boundedness of the $(p,p)$-Carleson map, that $E$ has type $p$.
    
    To obtain vectors $y_1,\ldots ,y_N \in E$ 
    (which we choose later) as dyadic averages of a function $f$ we 
    define $f(\xi) = 0$ for $\xi\in [0,2^{-N})$ and
    \begin{equation*}
      f(\xi) = 2y_j - y_{j+1} \quad \text{for} \quad \xi\in [2^{-j}, 2^{-j+1}), \quad j=1,\ldots , N,
    \end{equation*}
    where $y_{N+1} = 0$. Now, for $\xi\in [0,2^{-j+1})$ with $j=1,\ldots , N$ we have
    \begin{equation*}
      E_{j-1}f(\xi) = 2^{j-1} \sum_{k=j}^N 2^{-k} (2y_k - y_{k+1}) 
      = 2^{j-1} \Big( \sum_{k=j}^N 2^{-k+1}y_k - \sum_{k=j+1}^{N+1} 2^{-k+1}y_k \Big) = y_j ,
    \end{equation*}
    while $E_jf(\xi) = 0$ for $\xi\in [0,2^{-j})$ with $j \geq N$.
    
    A suitable choice of $\theta$ guarantees that
    the averages $E_jf$ need to be considered only on the intervals
    of the form $[0,2^{-j})$. Indeed, we define
    \begin{equation*}
      \theta_j = 2^{(N-j-1)/p} 1_{[0,2^{-N})} , \quad j=0,1,\ldots , N-1 ,
    \end{equation*}
    so that whenever $0 \leq m \leq N-1$, we have
    \begin{align*}
      \int_0^{2^{-m}} \E \Big| \sum_{j=m}^{N-1} \varepsilon_j \theta_j (\xi ) \Big|^p \D\xi
      &= \int_0^{2^{-m}} \E \Big| \sum_{j=m}^{N-1} 
      \varepsilon_j 2^{(N-j-1)/p} 1_{[0,2^{-N})}(\xi) \Big|^p \D\xi \\
      &= 2^{-N} \E \Big| \sum_{j=m}^{N-1} \varepsilon_j 2^{(N-j-1)/p} \Big|^p \\
      &\lesssim 2^{-N} \Big( \sum_{j=m}^{N-1} 2^{2(N-j-1)/p} \Big)^{p/2} \\
      &= \Big( \sum_{j=m}^{N-1} 2^{-2(j+1)/p} \Big)^{p/2} \lesssim 2^{-m} ,
    \end{align*}
    where Khintchine's inequality
    (the scalar version of Khintchine-Kahane inequality) was used in the third step.
    Thus $\| \theta \|_{\text{Car}^p} \lesssim 1$ independently of $N$.
    
    The choice $y_j = 2^{j/p}x_j$ now gives
    \begin{align*}
      \int_0^1 \E \Big\| \sum_{j=0}^{\infty} \varepsilon_j \theta_j(\xi) E_jf(\xi) \Big\|^p \D\xi
      &= \int_0^1 \E \Big\| \sum_{j=1}^N 
      \varepsilon_j 2^{(N-j)/p} 1_{[0,2^{-N})} (\xi) 2^{j/p}x_j \Big\|^p \D\xi \\
      &= 2^{-N} \E \Big\| \sum_{j=1}^N \varepsilon_j 2^{N/p} x_j \Big\|^p \\
      &= \E \Big\| \sum_{j=1}^N \varepsilon_j x_j \Big\|^p ,
    \end{align*}
    and all that remains is to calculate the norm of $f$:
    \begin{equation*}
      \| f \|_{L^p(E)} = \Big( \sum_{j=1}^N 2^{-j} \| 2y_j - y_{j+1} \|^p \Big) ^{1/p}
      \lesssim \Big( \sum_{j=1}^N 2^{-j} \| y_j \|^p \Big)^{1/p} 
      = \Big( \sum_{j=1}^N \| x_j \|^p \Big)^{1/p} .
    \end{equation*}
    We have shown that
    \begin{equation*}
      \Big( \E \Big\| \sum_{j=1}^N \varepsilon_j x_j \Big\|^p \Big)^{1/p} 
      \lesssim \Big( \sum_{j=1}^N \| x_j \|^p \Big)^{1/p} ,
    \end{equation*}
    which by Khintchine-Kahane inequality guarantees that $E$ has type $p$.

     $\eqref{XhasRMFandType}\Rightarrow\eqref{ppCarGeneral}$:
    Suppose that $\mathcal{X}$ has RMF, $E$ has type $p$ and let $\theta$ be a 
    $p$-Carleson family with respect to any filtration on any $\sigma$-finite measure space. 
    Since the Rademacher maximal operator maps $L^p(\mathcal{X})$
    boundedly to $L^p$ and $E$ has type $p$ (and hence $E$ does not contain $c_0$), we can apply Lemma \ref{mainlemma} to obtain
    \begin{equation*}
      \| \Theta f \|_{L^p(\text{Rad}(E))}
      \lesssim \| \theta \|_{\text{Car}^p} \| M_Rf \|_{L^p}
      \lesssim \| \theta \|_{\text{Car}^p} \| f \|_{L^p(\mathcal{X})} .
    \end{equation*}
    The $(p,p)$-Carleson map is thus well-defined and bounded.
  \end{proof}

\begin{remark}
Note that $p$ cannot be greater than $2$ in Theorem \ref{mainthm2}. It is shown in Kemppainen \cite{RMF}
that every space with RMF has non-trivial type and so the conditions in Theorem \ref{mainthm2} always hold
for some $p > 1$.
\end{remark}


\section*{Acknowledgements}

Tuomas Hyt\"onen is supported by the Academy of Finland through projects 130166 ``$L^p$ Methods in Harmonic Analysis'' and 133264 ``Stochastic and Harmonic Analysis: Interactions and Applications''.
Mikko Kemppainen gratefully acknowledges the support from The Finnish National Graduate School in Mathematical Analysis and Its Applications.

\bibliographystyle{plain}
\bibliography{viitteet}

\end{document}